\documentclass{amsart}

\usepackage[utf8]{inputenc}
\usepackage{amssymb}
\usepackage[leqno]{amsmath}
\usepackage{latexsym}
\usepackage{amsthm}
\usepackage{enumitem}
\usepackage{graphicx}
\usepackage{epstopdf}
\usepackage{bigints}
\usepackage[all,cmtip]{xy}
\usepackage{esint}
\setlist[itemize]{leftmargin=4ex}

\theoremstyle{plain}
\newtheorem{theorem}{Theorem}

\newtheorem{lemma}[theorem]{Lemma}
\newtheorem{proposition}[theorem]{Proposition}
\newtheorem{definition}[theorem]{Definition}

\theoremstyle{remark}
\newtheorem{remark}[theorem]{Remark}

\numberwithin{equation}{section}
\numberwithin{theorem}{section}
\newcommand{\App}{P}                                           %Approssimazione (dipendente dalla soluzione)
\newcommand{\app}{M}										   %Approssimazione (dipendente dal termine noto)
\newcommand{\Ritz}{\Pi}                                          %Proiezione di Ritz
\newcommand{\id}{\mathrm{Id}}                                  %Identità

\newcommand{\norm}[1]{\| #1 \|}                     %Norma due 
\newcommand{\opnorm}[3]{\norm{#1}_{\ifx#2#3\mathcal{L}(#2)\else\mathcal{L}(#2,#3)\fi}}        %Norma operatoriale
\newcommand{\R}{\mathbb{R}}                                   %Numeri reali

\newcommand{\clsint}[2]{[{#1},{#2}]}

\newcommand{\aext}{\widetilde{a}}                              %Estensione della forma bilineare a
\newcommand{\bext}{\widetilde{b}}                              %Estensione della forma bilineare b
\newcommand{\Vext}{{\widetilde{V}}}                              %Spazio esteso V+S
\newcommand{\vext}{\widetilde{v}}
\newcommand{\Appext}{\widetilde{\App}}                         %Estensione dell'approssimazione P

\DeclareMathOperator*{\infimum}{inf\vphantom{p}}               %inf
\newcommand{\Cqopt}{C_{\mathrm{qopt}}}                         %Costante di quasi-ottimalità
\newcommand{\Cstab}{C_{\mathrm{stab}}}                         %Costante di stabilità
\begin{document}

\title[Quasi-optimal nonconforming methods I]
{Quasi-optimal nonconforming methods for symmetric elliptic problems.
 I -- Abstract theory }

\author[A.~Veeser]{Andreas Veeser}
\author[P.~Zanotti]{Pietro Zanotti}

%\keywords{Quasi-optimality, Galerkin methods, nonconforming methods, discontinuous Galerkin}
%\subjclass[1991]{65N12, 65N15, 65N30}
%\date{  }

\begin{abstract}
We consider nonconforming methods for symmetric elliptic problems and characterize their quasi-optimality in terms of suitable notions of stability and consistency. The quasi-optimality constant is determined and the possible impact of nonconformity on its size is quantified by means of two alternative consistency measures. Identifying the structure of quasi-optimal methods, we show that their construction reduces to the choice of suitable linear operators mapping discrete functions to conforming ones.  Such smoothing operators are devised in the forthcoming parts of this work for various finite element spaces.
\end{abstract}

\maketitle

\section{Introduction}
%
%
%

% Ritz-Galerkin and Cea
Consider an elliptic boundary value problem, which can be cast in the abstract form
\begin{equation}
\label{sym-ell-prob}
 \text{find } u \in V \text{ such that }
 \forall v \in V \;\; a(u,v) = \langle \ell,v \rangle,
\end{equation}
where the bilinear form $a$ is a scalar product on the linear function space $V$. The Ritz-Galerkin method defines an approximation to $u$ as the solution $U$ of the problem where the infinite-dimensional space $V$ is replaced by a finite-dimensional subspace $S \subseteq V$.  C\'ea's lemma~\cite{Cea:64} reveals that $U$ is the best approximation to $u$ in $S$ with respect to the norm induced by $a$.  Remarkably, this holds irrespective of the regularity of the exact solution $u$.  In other words: the Ritz-Galerkin method is always optimal in $S$ with respect to the energy norm.

% Petrov-Galerkin and quasi-optimality
There are various generalizations of C\'ea's lemma. For Petrov-Galerkin methods applied to well-posed problems, Babu\v{s}ka \cite{Babuska:70} has shown the quasi-optimality property
\begin{equation}
\label{qo}
 \forall u \text{ solutions}
\quad
 \norm{u-U}
 \leq
 \Cqopt \inf_{s \in S} \norm{u-s}
\end{equation}
and, recently, Tantardini and Veeser \cite{Tantardini.Veeser:16} have shown that the best constant is
\begin{equation*}
 \Cqopt
 =
 \sup_{\sigma \in \Sigma}
  \frac{ \sup_{\norm{v} = 1} b(v,\sigma) }{ \sup_{\norm{s}=1} b(s,\sigma) },
\end{equation*}
where $b$ is the underlying bilinear form, $v$, $s$, and $\sigma$ vary, respectively, in the continuous trial space, the discrete trial space and the discrete test space.  This provides a rather general but still very strong result when the discrete spaces are conforming, that is, are subspaces of their continuous counterparts.

% nonconforming Galerkin methods
For classical nonconforming finite element methods (NCFEM) like the Crouzeix-Raviart or the Morley method and for Discontinuous Galerkin (DG) methods, such a strong result is not available, to our best knowledge.  Here the so-called second Strang lemma \cite{Berger.Scott.Strang:72} or variants serve as a replacement for C\'ea's lemma and the bound of the term associated with the consistency error is problematic.  It involves extra regularity,
\begin{itemize}
\item either of the solution $u$, which has to be taken from a strict compact subset of $V$, see, e.g., Brenner/Scott \cite{Brenner.Scott:08} and Di~Pietro/Ern \cite{Ern.DiPietro:12},
\item or, in the medius analysis initiated by Gudi \cite{Gudi:10}, of the load term $\ell$, which has to be taken from a strict compact subset of $V'$; see Brenner \cite{Brenner:15}. 
\end{itemize}
This extra regularity then obstructs a further bound by the best approximation error with respect to the energy norm in order to conclude quasi-optimality.

However, nonconforming discrete spaces are of interest because the `rigidity' of their conforming counterparts may cause problems in approximation, see, e.g., de~Boor/DeVore \cite{DeBoor.DeVore:83} and Babu\v{s}ka/Suri \cite{Babuska.Suri:92}, in stability, see Scott/Vogelius \cite{Scott.Vogelius:85}, or in accommodating structural properties like conservation.

% purpose of this article
\smallskip This article is the first in a project to close the gap of missing quasi-optimality for nonconforming methods. Here we consider continuous problems of the form \eqref{sym-ell-prob}, together with a rather big class of nonconforming methods.  This class contains in particular classical NCFEM, DG and other interior penalty methods. 

Our first main result states that quasi-optimality as in \eqref{qo} is equivalent to full algebraic consistency and full stability. Full algebraic consistency means that, whenever the exact solution happens to be in the discrete space, it is also the discrete solution. Notice that this is a quite weak property if the conforming part $S \cap V$ of the discrete space is small.  Full stability means that the discrete problem is stable for all loads, irrespective of their regularity. Moreover, we show that full stability holds if and only if the discrete problem reads
\begin{equation*}
 \text{find } U \in S \text{ such that }
 \forall \sigma \in S \;\; b(U,\sigma) = \langle \ell, E\sigma \rangle
\end{equation*}
where $b$ is the discrete bilinear form and $E:S \to V$ is a linear map, called smoother, and defined on the whole discrete space $S$.  Notice that, usually, nonconforming methods are used without a smoother and so full stability does not hold. It is thus not a surprise that previous results did not establish quasi-optimality with respect to the energy norm. Nonconforming methods with smoothing can be found in Arnold and Brezzi \cite{Arnold.Brezzi:85}, which observes increased stability, Brenner and Sung \cite{Brenner.Sung:05}, which presents fully stable methods, and Badia et al.\ \cite{Badia.Codina.Gudi.Guzman:14}, which contains also a partial quasi-optimality result. 

As a second main result, we determine the quasi-optimality constant, i.e.\ the best constant in \eqref{qo}, for a quasi-optimal nonconforming method:
\begin{equation*}
 \Cqopt
 =
 \sup_{\sigma \in S}
  \frac{ \sup_{\norm{v+s} = 1} a(v,E\sigma)+b(s,\sigma) }{ \sup_{\norm{s}=1} b(s,\sigma) }.
\end{equation*}
Notice that the enumerator handles the nonconformity by an extension interweaving data from the continuous and the discrete problem.  Moreover, we can determine $\Cqopt$ by two consistency measures generalizing algebraic consistency: one incorporating stability, one essentially independent of stability. 

These results reduce the construction of quasi-optimal nonconforming methods for \eqref{sym-ell-prob} to devising suitable smoothers $E$.  This is established for various nonconforming finite element spaces in our forthcoming works \cite{Veeser.Zanotti:17p2,Veeser.Zanotti:17p3}.

\section{Setting, stability and consistency}
\label{S:setup}
%
%
% intro
This section sets up the notations and notions for our analysis, individuating concepts of stability and consistency that are necessary for quasi-optimality. 

\subsection{Symmetric elliptic problems and nonconforming methods}
\label{S:setting}
%
%
% abstract continuous problem
We introduce the abstract boundary value problem and then a class of nonconforming methods, sufficiently large to host our discussion.

Let $V$ be an infinite-dimensional Hilbert space with scalar product $a(\cdot, \cdot)$ and \emph{energy norm} $\norm{\cdot} = \sqrt{a(\cdot,\cdot)}$.  Moreover, let $V'$ be the topological dual space of $V$, denote by $\left\langle \cdot, \cdot\right\rangle$ the pairing of $V$ and $V'$ and endow $V'$ with the \emph{dual energy norm}
$\norm{\ell}_{V'} := \sup_{v\in V, \norm{v} = 1} \langle\ell,v\rangle$. 
We consider the following \emph{`continuous' problem}: given $\ell \in V'$, find $u\in V$ such that
\begin{equation}
\label{ex-prob}
\forall v \in V
\quad
a(u, v) 
= 
\langle \ell, v \rangle.
\end{equation} 
In view of the Riesz representation theorem, this problem is well-posed in the sense of Hadamard and well-conditioned. In fact, if $A:V \to V'$, $v \mapsto a(v, \cdot)$ is the Riesz isometry of $V$, we have $u=A^{-1}\ell$ with
\begin{equation}
\label{isometry}
\norm{u} = \norm{\ell}_{V'}.
\end{equation}

% (entire) methods
Given a generic functional $\ell \in V'$, we are interested in `computable' approximations of the solution $u$ in \eqref{ex-prob}.  In other words, we are interested in approximating the linear operator $A^{-1}$  suitably.  Since $A^{-1}$ is bounded, one may want to approximate it by linear operators that are bounded, too.  However, in order to embed also existing methods in our setting, we consider more general linear operators $\app$, possibly unbounded, with finite-dimensional range $R(\app)$ and domain $D(\app)$ that is dense in $V'$.  We say that $\app$ is \emph{entire} whenever it can be directly applied to every instance of the continuous problem: $D(M) = V'$.  

% discrete space and problem
We shall analyze methods that build upon the variational structure of \eqref{ex-prob} in the following manner. Let $S$ be a nontrivial, finite-dimensional linear space, which will play the role of $V$. %in the discrete counterpart of \eqref{ex-prob}. 
We write $\left\langle \cdot, \cdot\right\rangle$ also for the pairing of $S$ and $S'$. Notice that we do not require $S \subseteq V$. As a consequence, $\langle\ell,\sigma\rangle$ and $a(s,\sigma)$ may be not defined for some $\ell\in V'$ and $s,\sigma\in S$.  We therefore introduce an operator $L: D(L) \subseteq V' \to S'$ and a counterpart $b : S \times S \to \R$ of $a$ and require:
\begin{itemize}
\item $L$ is linear, (possibly) unbounded, and densely defined, 
\item $b$ is bilinear and nondegenerate in that, for any $s\in S$, the property
$b(s,\sigma) = 0$ for all $\sigma \in S$ entails $s = 0$.
\end{itemize}
A method $\app$ with domain $D(\app) = D(L)$ is then defined by the following \emph{discrete problem}: given $\ell \in D(\app)$, find $\app\ell \in S$ such that
\begin{equation}
\label{disc-prob}
\forall \sigma \in S
\quad
b(\app \ell, \sigma)
=
\langle L\ell, \sigma \rangle.
\end{equation}

\begin{remark}[Computing discrete solutions]
\label{R:ComputingDiscreteSolutions}
If $\varphi_1,\dots,\varphi_n$ is some basis of $S$, \eqref{disc-prob} can be reformulated as a uniquely solvable linear system for the coefficients of $\app\ell$ with respect to $\varphi_1,\dots,\varphi_n$.  Consequently,
$\app\ell$ is computable, whenever  $b(\varphi_j,\varphi_i)$ and $\langle L\ell,\varphi_i\rangle$ can be evaluated for $i,j=1,\dots,n$.  Of course, it is desirable that the number of operations to compute $\app\ell$ is of optimal order $O(n)$.  A necessary condition for this is that the total number of operations for the aforementioned evaluations is of order $O(n)$.
\end{remark}

% (non)conforming variational method
Methods $\app$ with the discrete problem \eqref{disc-prob} are given by the triplet $(S,b,L)$, whence we shall write also $\app=(S,b,L)$.  They may be called \emph{nonconforming linear variational methods} or, shortly, \emph{nonconforming methods}.  An important subclass are the \emph{conforming} ones, where the discrete space is contained in the continuous one: $S \subseteq V$. (As for the common usage of `unbounded' and `bounded' in operator theory, our usage of `nonconforming' and `conforming' is slightly inconsistent in that a conforming method is also nonconforming.)  Conformity allows choosing $b$ and $L$ by means of simple restriction:
\begin{equation}
\label{ConformingGalerkin}
b = a_{|S \times S}
\quad\text{and}\quad
\forall \ell \in V' \;\; L\ell = \ell_{|S}.
\end{equation}
In this case \eqref{disc-prob} is a (conforming) \emph{Galerkin method}. Truly nonconforming examples are DG methods and classical NCFEM.

% representation of such methods and associated approximation operator
Introducing the invertible map $B:S \to S'$, $s \mapsto b(s,\cdot)$, the method $\app$ is represented by the composition
\begin{equation}
\label{M=}
\app = B^{-1} L.
\end{equation}
%Moreover,
Although the target function $u$ is usually unknown, the approximation operator
\begin{equation}
\label{Def-P}
\App := \app A = B^{-1} L A 
\end{equation}
with domain $D(\App) := A^{-1} D(\app)$ in $V$ will turn out to be a useful tool. Figure \ref{F:EntireNonconformingMethods} illustrates our setting in a commutative diagram for the special case of an entire method.

\begin{remark}[$S$ and surjectivity of $L$]
\label{R:surjectivity-of-L}
If $L$ is a linear, unbounded, densely defined operator from $V'$ to $S'$, we have $R(\app) \subseteq S$, with equality if and only if $L$ is surjective.  In addition, if $R(\app)$ is a proper subset of $S$, elementary linear algebra allows to reformulate $\app$ as a method over $R(\app)$.  Consequently, there is some ambiguity in the choice of $S$ if $L$ is not surjective and a slight abuse of notation in writing $\app=(S,b,L)$. 
\end{remark}

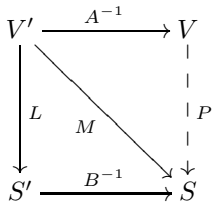
\begin{figure}	
	\[
	\xymatrixcolsep{4pc}		% Size of column separation
	\xymatrixrowsep{4pc}		% Size of row separation
	\xymatrix{
		V'   									% Space 1
		\ar[d]^{L}
		\ar[r]^{A^{-1}}       % Right arrow
		\ar[rd]_{M} 					% Right-down arrow 
		& V 									% Space 2
		\ar@{-->}[d]^{\App} \\		% Down arrow (with label)
		S'										% Space 3
		\ar[r]^{B^{-1}}				% Right arrow
		& S}               		% Space 4
	\]
	\caption{\label{F:EntireNonconformingMethods}	Commutative diagram with solution operator $A^{-1}$, entire nonconforming variational method $\app$ given by $S$, $B$ and $L$, as well as induced approximation operator $\App$.}
\end{figure}

\subsection{Defining quasi-optimality, stability and consistency}
\label{S:stab-cons}
We now define the key notions of our analysis for nonconforming methods.

% extended energy norm
For each $\ell\in V'$, a nonconforming variational method $\app=(S,b,L)$ chooses an element of $S$ in order to approximate $u=A^{-1}\ell$.  To assess the quality of this choice, we assume that $a$ can be extended to a scalar product $\aext$ on $\Vext:=V+S$ and consider the \emph{extended energy norm}
\begin{equation*}
\norm{\cdot} := \sqrt{\aext(\cdot, \cdot)}
\quad\text{on }\Vext,
\end{equation*}
with the same notation as for the original one. Observe that $V$ and $S$ are closed subspaces of $\Vext$.

% quasi-optimality
The best approximation error within $S$ to some function $v \in V$ is then given by $\inf_{s\in S} \norm{v-s}$.  Of course, it is desirable that a method is uniformly close to this benchmark, i.e.\ there holds an inequality that essentially reverses
\begin{equation*}
\forall u \in D(\App)
\qquad
\inf_{s\in S} \norm{u-s}
\leq
\norm{ u - \App u }.
\end{equation*}
\begin{definition}[Quasi-optimality]
\label{D:qopt}	
A nonconforming variational method $\app$ with discrete space $S$ and approximation operator $\App$ is \emph{quasi-optimal} whenever there exists a constant $C\geq 1$ such that
\[
\forall u \in D(\App)
\qquad
\norm{ u - \App u }
\leq
C \inf_{s\in S} \norm{u-s}.
\]
The quasi-optimality constant $\Cqopt$ of $M$ is then the smallest constant with this property.
\end{definition}

% quasi-optimality for conforming methods
C\'ea's lemma \cite{Cea:64} shows that conforming Galerkin methods for \eqref{ex-prob} are quasi-optimal with $\Cqopt=1$ and that the associated approximation operator $\App=\app A$ is the bounded linear $a$-orthogonal projection (or idempotent) onto $S$: in fact, we have the celebrated Galerkin orthogonality
\begin{equation}
\label{Galerkin-orthogonality}
\forall u \in V, \sigma\in S \subseteq V
\qquad
a(u-\App u,\sigma) = 0.
\end{equation}
Before analyzing which of these properties still hold in the general case, let us discuss some necessary conditions for quasi-optimality and their consequences.

\begin{remark}[Quasi-optimal needs entire]
\label{R:qopt->entire}
Let $\App$ be the approximation operator of a quasi-optimal method $\app$.  Observe that the best error $\inf_{s\in S} \norm{\cdot-s}$ is a Lipschitz continuous function on $V$.  Therefore, quasi-optimality implies that also $\id_V-\App$ and $\App$ are Lipschitz continuous.  Since $D(\App)$ is dense in $V$ and $S$ complete, the operator $\App$ thus extends to $V$ in a continuous and unique manner.  As a consequence, $\app$ extends to $V'$ in a continuous and unique manner.   In other words: ignoring the aspect of computability, only entire methods can be quasi-optimal. 
\end{remark}

Notice that most classical NCFEM and DG methods are not defined as entire. Consequently, the simple observation in Remark \ref{R:qopt->entire} questions that these methods can be quasi-optimal. This doubt will be confirmed in Remark~\ref{R:failure-of-idS} below.

%
% stability for nonconforming methods
%
\medskip Generally speaking, stability is associated with the property that small input perturbations result in small output perturbations. The form of the discrete problem \eqref{disc-prob} suggests adopting the viewpoint that input is taken from a subset of $V'$. Since \eqref{disc-prob} is linear, stability then amounts to some operator norm of $M$.  Notice that this differs from the common viewpoint that stability is connected solely with an operator norm of $B^{-1}$, i.e.\ taking input from $S'$. In the following definition, we consider perturbations and measure them as suggested by the setting of the continuous problem.

\begin{definition}[Full stability]
\label{D:stab}
We say that $M$ is \emph{fully stable} whenever $D(M)=V'$ and, for some constant $C\geq0$, we have
\begin{equation*}
%\label{D(M)-stability}
%
	\forall \ell \in V'
	\qquad
	\norm{\app \ell}
	\leq
	C \norm{\ell}_{V'}.
\end{equation*}
The smallest such constant is the stability constant $\Cstab$ of $\app$.
\end{definition}

Full stability may go beyond the need for practical computations, but it relates to the previous notions in the following manner.

\begin{remark}[Fully stable, quasi-optimal and entire]
\label{R:FullStability}
The approximation operator $P$ of a quasi-optimal method satisfies
\begin{equation*}
	\norm{Pu}
	\leq
	\norm{u} + \norm{Pu-u}
	\leq
	(1+\Cqopt) \norm{u}
	=
	(1+\Cqopt) \norm{Au}_{V'}
\end{equation*}
for all $u\in V$, using $0\in S$, \eqref{isometry} and Remark \ref{R:qopt->entire}. In view of \eqref{Def-P}, full stability is thus necessary for quasi-optimality.  Furthermore,  full stability itself requires that the method is entire in the vein of Remark \ref{R:qopt->entire}.
\end{remark}

%
% consistency for nonconforming methods
%
\medskip Roughly speaking, consistency measures to what extent the exact solution verifies the discrete problem.
%
% obstructions and regularity
To this end, one usually substitutes in the discrete problem the discrete solution by the exact one and investigates a possible defect.  Here nonconformity entails that the forms $b$ and $L$ cannot be defined by simple restriction and so creates the following issues concerning trial and test space:
\begin{itemize}
	\item In which sense can we plug a generic exact solution $u$ into the discrete problem? Does this require an extension of $b$ or a representative of $u$ in $S$? 
	\item How do we relate the condition associated with a nonconforming test function $\sigma \in S\setminus V$ in \eqref{disc-prob} to the conditions given by the continuous test functions in \eqref{ex-prob}?
\end{itemize}
These issues are usually tackled with the help of regularity assumptions on the exact solution, see, e.g., Arnold et al.\ \cite{Arnold.Brezzi.Cockburn.Marini:02}, or only on data, see Gudi \cite{Gudi:10}.   The following definition takes a different approach within our non-asymptotic setting. 
\begin{definition}[Full algebraic consistency]
\label{D:cons}
The method $\app$ is \emph{fully algebraically consistent} whenever $D(M)=V'$ and 
\begin{equation}
\label{Consistency}
	\forall u\in V \cap S, \sigma \in S
	\quad
	b(u,\sigma) = \langle LAu,\sigma \rangle.
\end{equation}
\end{definition}
Conforming Galerkin \eqref{ConformingGalerkin} methods are fully algebraically consistent.  Let us discuss further aspects of full algebraic consistency.

\begin{remark}[Full algebraic consistency and approximation operator]
\label{R:reformulations-of-consistency}
In view of the discrete problem \eqref{disc-prob} and the definition \eqref{Def-P} of the approximation operator, \eqref{Consistency} is equivalent to $b(u - \App u,\sigma) = 0$ for all $u\in V \cap S, \sigma \in S$.
Since $b$ is nondegenerate, the consistency condition \eqref{Consistency} is therefore equivalent to
\begin{equation}
\label{P|V cap S}
	\forall u\in V \cap S
\quad
	\App u = u.
\end{equation}
In other words: full algebraic consistency means that whenever the exact solution is discrete, it is the discrete solution.  The advantage of \eqref{Consistency} is that it is directly formulated in terms of the originally given data $A$, $S$, $b$ and $L$.  In Lemma \ref{L:ConsistencyWithExtension} and Theorem \ref{T:qopt-smoothing} below, we will present further equivalent formulations.
\end{remark}

% quasi-optimality requires full consistency
\begin{remark}[Quasi-optimal needs fully algebraically consistent]
\label{R:QuasiOptRequiresFullConsistency}
In light of Remark~\ref{R:qopt->entire}, a quasi-optimal method $M$ is entire and so its approximation operator $P$ is defined on all $V$.  For any $u \in V\cap S$, the best error in $S$ vanishes and so $Pu=u$.  Consequently, $M$ is fully algebraically consistent.
\end{remark}

% link with classical consistency
Definition \ref{D:cons} involves only exact solutions from the discrete space $S$, which may be a quite small set. Indeed, for example, when applying the Morley method to the biharmonic problem, the intersection $S \cap V$ has poor approximation properties for certain mesh families; see \cite[Theorem 3]{DeBoor.DeVore:83} and \cite[Remark~3.11]{Veeser.Zanotti:17p2}.  Other consistency notions of algebraic type involving more exact solutions may thus appear stronger than Definition \ref{D:cons}.  The following lemma sheds a different light on this.

\begin{lemma}[Full algebraic consistency with extension]
\label{L:ConsistencyWithExtension}
Let the method $M$ be fully algebraically consistent and set $\Vext := V + S$. Then there exists a unique bilinear form $\bext$ that extends $b$ as well as $\langle LA\cdot,\cdot\rangle$ on $\Vext \times S$. 
\end{lemma}

\begin{proof}
Observe that the left-hand side of \eqref{Consistency} is defined for all $u \in S$, while its right-hand side is defined in particular for all $u \in V$.  We exploit this in order to extend $b$. Given $\vext \in \Vext$ and $\sigma \in S$, we write $\vext = v + s$ with $v \in V$ and $s \in S$ and set
\begin{equation}
\label{bext}
	\bext(\vext,\sigma)
	:=
	\langle LAv,\sigma \rangle + b(s,\sigma).
\end{equation}
Thanks to \eqref{Consistency}, $\bext$ is well-defined.  Indeed, if $v_1 + s_1 = v_2 + s_2$ with $v_1, v_2 \in V$ and $s_1, s_2 \in S$, we have $v_1-v_2 = s_2 - s_1 \in V \cap S$ and therefore \eqref{Consistency} yields $\langle LA(v_1-v_2),\sigma \rangle = - b(s_1-s_2,\sigma)$, which in turn ensures
\begin{equation*}
	\langle LA v_1, \sigma \rangle + b(s_1,\sigma)
	=
	\langle LA v_2, \sigma \rangle + b(s_2,\sigma).
\end{equation*}
To show uniqueness of the extension, let $\widetilde{\beta}$ be another common extension of $b$ and $\langle LA\cdot, \cdot \rangle$. Given $\vext \in \Vext$ and $\sigma \in S$, we write $\vext = v + s$ with $v \in V$ and $s \in S$ as before and infer
\begin{equation*}
	\widetilde{\beta}(\vext,\sigma)
	=
	\widetilde{\beta}(v,\sigma) + \widetilde{\beta}(s,\sigma)
	=
	\langle LAv,\sigma \rangle + b(s,\sigma)
	=
	\bext(\vext,\sigma)
	\end{equation*}
and the proof is complete.
\end{proof}
 
% comparing with classical consistency
Notice that full algebraic consistency differs from the usual consistency, as, e.g. in Arnold \cite{Arnold:15} also for the following aspects: on the one hand, it is stronger in that it requires an algebraic identity instead of a limit.  On the other hand, it does not involve approximation properties of the underlying discrete space.  In fact, our purpose here is to identify the part of consistency that is necessary for quasi-optimality. As a consequence, algebraic consistency and stability alone are not sufficient for convergence.
  
% nonconforming Galerkin methods
\medskip Let us conclude this section by introducing a subclass of natural candidates for fully algebraically consistent methods.  A method $\app = (S,b,L)$ is a \emph{nonconforming Galerkin method} whenever
\begin{equation}
\label{NonConformingGalerkinMethod}
b_{|S_C \times S_C} = a_{|S_C \times S_C}
\quad\text{and}\quad
\forall \ell \in D(L) \;\;
L\ell_{|S_C} = \ell_{|S_C},
\end{equation}
where $S_C = S \cap V$ is the conforming subspace of the discrete space $S$. Thus, a nonconforming Galerkin method is constrained by restriction where applicable.  Notice that:
\begin{itemize}
	\item In contrast to conforming Galerkin methods, nonconforming ones are not completely determined by the continuous problem and the discrete space.
	\item The condition \eqref{NonConformingGalerkinMethod} readily yields
	\begin{equation*}
	\forall u, \sigma \in S \cap V
	\quad
	b(u, \sigma) = \langle LAu, \sigma \rangle,
	\end{equation*}
	which is weaker than full algebraic consistency in that less test functions are involved.
\end{itemize}
For example, classical NCFEM, DG and $C^0$ interior penalty methods are nonconforming Galerkin methods. 

\section{Characterizing quasi-optimality}
\label{S:qopt}
The purpose of this section is twofold. First, we show that full algebraic consistency and full stability are not only necessary but also sufficient for quasi-optimality. Second, we assess the possible impact of nonconformity on the quasi-optimality constant.
\subsection{Quasi-optimality and extended approximation operator}
\label{S:to-qopt}
%
% intro
To show that full algebraic consistency and full stability imply quasi-optimality,
we start with the following short proof of a `partial' quasi-optimality, which motivates a new tool for the analysis of nonconforming methods. 

% motivation for extended approximation operator
Assume that $\App$ is the approximation operator of a fully algebraically consistent and a fully stable method. Rewriting \eqref{P|V cap S} as
\begin{equation}
\label{I-P-and-ScapV}
 \forall v\in V, s \in S \cap V
\quad
 v-Pv
 =
 (\id_V-P)(v-s)
\end{equation}
and exploiting that full stability entails the boundedness of $\App$, we can deduce quasi-optimality with respect to the conforming part $S\cap V$ of the discrete space $S$:
\begin{equation*}
 \norm{v-Pv}
 \leq
 \opnorm{\id_V-\App}{V}{S} \inf_{s \in S\cap V} \norm{v-s}.
\end{equation*}
Note that we do not obtain quasi-optimality with respect to the whole discrete space, just because $\App s = s$ is not available for general $s\in S$.  In particular, $\App s$ is not defined for general $s \in S$. We therefore explore an appropriate extension of $\App$.

For this purpose, we use the following facts on linear projections; cf., e.g., Buckholtz \cite{Buckholtz:00}. Let $K$ and $R$ be subspaces of a Hilbert space $H$ with scalar product $(\cdot,\cdot)_H$ and induced norm $\norm{\cdot}_H$. The spaces $K$ and $R$ provide a direct decomposition of $H$, $H = K \oplus R$, if and only if there exists a unique linear projection $Q$ on $H$ with kernel $N(Q)=K$ and range $R(Q)=R$. Then $\id_H-Q$ is the linear projection with kernel $R$ and range $K$. As a consequence of the closed graph theorem,  $R$ and $K$ are closed if and only if $Q$ is bounded if and only if $\id_H - Q$ is bounded.
\begin{lemma}[Extended approximation operator]
\label{L:Pext}
Assume that	the approximation operator $\App$ verifies $\App|_{S \cap V} = \id_{S \cap V}$ and is bounded.  Then there exists a unique bounded linear projection $\Appext$ from $\Vext$ onto $S$ satisfying $\Appext_{|V} = \App$.
\end{lemma}
\begin{proof}
% uniqueness and existence of exteneded approximation operator
First, we observe that $\Appext$ has to satisfy
\begin{equation}
\label{properties-of-Pext}
	\Appext:\Vext \to S \text{ linear},
	\quad
	\Appext_{|V} = P
	\quad\text{and}\quad
	\Appext_{|S} = \id_S.
\end{equation}
Since $\Vext = V + S$, linear extension entails that there is at most one operator satisfying \eqref{properties-of-Pext} and we are thus led to consider the following definition: given $\vext\in\Vext$, choose $v \in V$ and $s \in S$ such that $\vext = v+s$ and set
\begin{equation}
\label{Pext}
	\Appext \vext
	:= 
	\App v + s.
\end{equation}
The assumption $\App_{|S \cap V} = \id_{S \cap V}$ means that the two identities in \eqref{properties-of-Pext} are compatible and so guarantees that $\Appext$ is well-defined; compare with the definition of $\bext$ in the proof of Lemma~\ref{L:ConsistencyWithExtension}. 

% boundedness
In order to show the boundedness of $\Appext$, we represent it in terms of $P$ and the following operators, corresponding to an appropriate choice of $v$ and $s$ in \eqref{Pext}.  Let $\Ritz_Y$ be the $\aext$-orthogonal projection onto $Y := (S\cap V)^\perp$ and let $Q$ be the linear projection on $Y$ with range $V\cap Y$ and kernel $S \cap Y$.  We then have
\begin{equation*}
 \Appext
 =
 P Q \Ritz_Y
 +
 (\id_Y - Q) \Ritz_Y
 +
 (\id_{\Vext} - \Ritz_Y)
 =
 P Q \Ritz_Y
 +
 \id_{\Vext} -  Q \Ritz_Y.
\end{equation*}
Since the subspaces $S$, $V$, and $Y$ are closed, the projections $\Ritz_Y$ and $Q$ are bounded.  Consequently, the boundedness of $\App$ implies the boundedness of its extension $\Appext$.
\end{proof}

% quasi-optimality by extended approximation operator
Using the extended approximation operator $\Appext$, the proof of the announced characterization of quasi-optimality is quite simple. Notice also that the quantitative aspect of our first main result highlights the importance of $\Appext$. 
\begin{theorem}[Characterization of quasi-optimality]
\label{T:qopt}
A nonconforming method is quasi-optimal if and only if it is fully algebraically consistent and fully stable. 

Moreover, for any quasi-optimal method, we have
\begin{equation*}
\Cqopt
=
\opnorm{\Appext}{\Vext}{\Vext} 
\end{equation*}
where $\Appext$ is the extended approximation operator from Lemma \ref{L:Pext}.
\end{theorem}

\begin{proof}
Remarks \ref{R:FullStability} and \ref{R:QuasiOptRequiresFullConsistency} show that quasi-optimality implies full algebraic consistency and full stability.

To show the converse, consider any fully algebraically consistent and fully stable nonconforming method. We simply follow the lines of the corresponding part of the proof of Tantardini/Veeser \cite[Theorem~2.1]{Tantardini.Veeser:16}, replacing $\App$ by $\Appext$ and exploiting the following generalization of \eqref{I-P-and-ScapV}:
\begin{equation}
\label{I-Appext}
\forall v\in V, s\in S
\quad
(\id_\Vext - \Appext) (v-s)
=
(\id_V - \App) v.
\end{equation} Given arbitrary $v\in V$ and $s\in S$, we thus derive
\begin{equation*}
	\norm{v-\App v}
	=
	\norm{(v-s) - \Appext(v-s)}
	\leq
	\opnorm{\id_{\Vext} - \Appext}{\Vext}{\Vext} \norm{v-s}.
\end{equation*}
Taking the infimum over all $s\in S$ and then the supremum over all $v\in V$, we obtain
\begin{equation}
\label{Copt<=}
	\Cqopt
	\leq
	\opnorm{\id_{\Vext} - \Appext}{\Vext}{\Vext}
\end{equation}
and see that $\app$ is quasi-optimal because $\Appext$ is bounded.

To verify, the identity for $\Cqopt$, let us first see that \eqref{Copt<=} is actually an equality. In fact, for $v \in V$ and $s \in S$, we derive
\begin{equation*}
\norm{(\id_{\Vext} - \Appext)(v+s)}
=
\norm{v-\App v}
\leq
\Cqopt \inf \limits_{\hat{s} \in S}\norm{v-\hat{s}}
\leq
\Cqopt
\norm{v+s}
\end{equation*}
using \eqref{I-Appext} again.  We thus obtain the converse to \eqref{Copt<=} by taking the supremum over all $v\in V$ and $s \in S$.

Moreover, since $\{0\}\subsetneq S\subsetneq \Vext$, the extended approximation operator $\Appext$ is a bounded linear idempotent with $0 \neq \Appext = \Appext^2 \neq \id_{\Vext}$ on the Hilbert space $\Vext$. We therefore can apply Buckholtz \cite[Theorem 2]{Buckholtz:00} or Xu/Zikatanov \cite[Lemma~5]{Xu.Zikatanov:03}  
and conclude 
\begin{equation}
\label{Cqopt=}
\Cqopt
=
\opnorm{\id_{\Vext} - \Appext}{\Vext}{\Vext}
=
\opnorm{\Appext}{\Vext}{\Vext}. \qedhere
\end{equation}
\end{proof}

Formula \eqref{Cqopt=} allows for the following geometric interpretation of the quasi-optimality constant.

\begin{remark}[Geometry of quasi-optimality constant]
	\label{R:geometry}
	Buckholtz \cite{Buckholtz:00} shows that the operator norm of a bounded projection $Q$ on a Hilbert space $H$ satisfies
	\begin{equation*}
	\opnorm{Q}{H}{H}
	=
	\frac{1}{\sin\theta}
	=
	\opnorm{\id_H - Q}{H}{H},
	\end{equation*}
	where $\theta$ is the angle between $K=N(Q)$ and $R=R(Q)$, that is, $\theta\in(0,\pi/2]$ and its cosine equals $\sup\{|\langle k,r \rangle_H| \mid k\in K, r\in R, \norm{k}_H=1, \norm{r}_H=1 \}$.  Notice that $N(\Appext) =  R(\id_\Vext-\Appext) = R(\id_V-\App)$, where the last identity follows from \eqref{I-Appext}.  Combining these two facts, we deduce
	\begin{equation}
	\label{|Pext|}
	\Cqopt
	=
	\opnorm{\Appext}{\Vext}{\Vext}
	=
	\frac{1}{\sin\alpha}
	\end{equation}
	where $\alpha$ is the angle between the discrete space $S$ and the range $R(\id_V-\App)$.
\end{remark}

Theorem \ref{T:Cqopt} reveals that the possibly weak full algebraic consistency is still enough consistency to ensure, together with stability, quasi-optimality. However, it does not control the size of the quasi-optimality constant.

\subsection{The quasi-optimality constant and two consistency measures}
\label{S:Cqopt}
Let $\App$ be the approximation operator of a quasi-optimal method. The fact that $\Appext$ is an extension of $\App$ readily yields
\begin{equation*}
\Cqopt
=
\opnorm{\Appext}{\Vext}{\Vext}
\geq
\opnorm{\App}{V}{S}
=
\Cstab,
\end{equation*}
where the last identity is due to isometry \eqref{isometry} of $A$. The possible enlargement of $\Cqopt$ with respect to $\Cstab$ is a new feature triggered by nonconformity. It is the purpose of the section to quantify this phenomenon.

Our key tool will be the following elementary lemma.
\begin{lemma}[Operator norm and restrictions]
\label{L:restrictions}
Let $T\in\mathcal{L}(H)$ be a bounded linear operator on a Hilbert space $H$ with scalar product $\langle\cdot,\cdot\rangle_H$ and induced norm $\norm{\cdot}_H$. If $Y$ is a linear closed subspace of $H$ and $Y^\perp$ is its orthogonal complement, we have
\begin{equation*}
  \max\{ C,\delta \}
  \leq
  \opnorm{T}{H}{H}
  \leq
  \sqrt{C^2+\delta^2}
\end{equation*}
with
\begin{equation*}
 C = \opnorm{T_{|Y}}{Y}{H}
\quad\text{and}\quad
 \delta = \opnorm{T_{|Y^\perp}}{Y^\perp}{H}.
\end{equation*}
\end{lemma} 

\begin{proof}
The lower bound immediately follows from the definition of the operator norm $\opnorm{T}{H}{H} = \sup_{\norm{x}_H=1} \norm{Tx}_H$. To verify the upper bound, let $x \in H$ be arbitrary and denote by $\pi_Y$ the orthogonal projection onto $Y$.  We have
\begin{equation}
\label{|Tx|<=;1}
\begin{aligned}
 \norm{Tx}_H^2
 &=
 \norm{T\pi_Yx}_H^2
  + 2 \big\langle T\pi_Yx, T(x-\pi_Yx) \big\rangle_H
  + \norm{T(x-\pi_Yx)}_H^2
\\
 &\leq
 C^2 \norm{\pi_Yx}_H^2
 + 2 C \delta\norm{\pi_Yx}_H \norm{x-\pi_Yx}_H
 + \delta^2 \norm{x-\pi_Yx}_H^2 
\end{aligned}
\end{equation}
in view of the bilinearity of the scalar product, the Cauchy-Schwarz inequality and the definitions of $C$ and $\delta$. Notice that
\begin{equation*}
  \norm{\pi_Yx}_H^2 + \norm{x-\pi_Yx}_H^2 = \norm{x}_H^2
\end{equation*}
thanks to the orthogonality of $\pi_Y$. Thus, if we write $\alpha = \norm{\pi_Yx}$, \eqref{|Tx|<=;1} becomes
\begin{equation*}
 \norm{Tx}_H^2
 \leq
 h(\alpha)^2
\quad\text{with}\quad
 h(\alpha)
 :=
 C\alpha + \delta \sqrt{1-\alpha^2},
\end{equation*}
which implies
\begin{equation*}
 \opnorm{T}{H}{H}
 \leq
 \max_{\clsint{0}{1}} h.
\end{equation*}
A straight-forward discussion of the function $h$ yields $\max_{\clsint{0}{1}} h = \sqrt{C^2 + \delta^2}$ and the upper bound is established, too.
\end{proof}

\begin{remark}[Sharpness of bounds via restrictions]
\label{R:sharpness-bounds-restriction}
Since
\begin{equation*}
 \max \{C,\delta \}
 \leq
 \sqrt{C^2 + \delta^2}
 \leq
 \sqrt{2}\max\{ C,\delta \},
\end{equation*}
the bounds in Lemma \ref{L:restrictions} miss an equality at most by the factor $\sqrt{2}$.  Let us see with two simple examples that, without additional information on $T$ and $Y$, we cannot improve on this.

First, consider $H = \R^2$, $T_1 = \id_{\R^2}$ and let $Y$ be any 1-dimensional subspace of $\R^2$. Obviously, we then have $\opnorm{ T_1 }{H}{H} = \opnorm{T_1{}_{|Y}}{Y}{H} = \opnorm{T_1{}_{|Y^\perp}}{Y^\perp}{H} = 1 $ and so
the lower bound becomes an equality, while the upper bound is strict.
%\begin{multline*}
% \max\{ \opnorm{T_1{}_{|Y}}{Y}{H}  , \opnorm{T_1{}_{|Y^\perp}}{Y^\perp}{H} \}
% =
% 1
% =
% \opnorm{ T_1 }{H}{H}
%\\
% <
% \sqrt{2}
% =
% \sqrt{ \opnorm{T_1{}_{|Y}}{Y}{H}^2 + \opnorm{T_1{}_{|Y^\perp}}{Y^\perp}{H}^2 }. 
%\end{multline*}

Second, consider $H = \R^2$ and let $T_2$ be the linear operator which is represented in the canonical basis of $\R^2$ by the Matlab matrix \texttt{1/2*ones(2)}. The operator $T_2$ is the orthogonal projection onto the diagonal $\{ (t,t) \mid t \in \R \}$, whence $\opnorm{ T_2 }{H}{H}=1$.  Finally, let $Y = \{ (0,t) \mid t \in \R \}$ be the ordinate.  Then the operator norms of $T_2$ restricted to $Y$ and $Y^\perp$ correspond to the Euclidean norms of the columns of the aforementioned matrix: $\opnorm{T_2{}_{|Y}}{Y}{H} = \opnorm{T_2{}_{|Y^\perp}}{Y^\perp}{H} = 1/\sqrt{2}$.  Consequently, here the upper bound is an equality, while the lower bound is strict.
\end{remark}

The fact that the extended approximation operator $\Appext$ is given on $S$ by the identity and on $V$ by $\App$ suggests two options for applying Lemma \ref{L:restrictions}: $Y=S$ and $Y=V$. We start with the first option, which leads to a consistency measure in the spirit of the second Strang lemma.

\begin{proposition}[Consistency mixed with stability]
\label{P:consistency-with-stability}
Let $\Ritz_S$ be the $\aext$-orthogonal projection onto $S$ and $\delta_V \geq 0$ be the smallest constant such that
\begin{equation*}
 \forall v \in V
\quad
 \norm{\Ritz_S v - P v}
 \leq
 \delta_V
 \norm{v - \Ritz_S v}.
\end{equation*}
Then the quasi-optimality constant is given by
$%\begin{equation*}
 \Cqopt = \sqrt{1+\delta_V^2}.
$%\end{equation*}
\end{proposition}

\begin{proof}
Owing to Theorem \ref{T:qopt}, we may show the claimed identity by verifying $\opnorm{\Appext}{\Vext}{\Vext} = \sqrt{1+\delta_V^2}$.
Applying Lemma \ref{L:restrictions} with $H=\Vext$, $T=\Appext$ and $Y=S$, we obtain
\begin{equation*}
 \opnorm{\Appext}{\Vext}{\Vext} \leq \sqrt{1+\delta^2}
\end{equation*}
with $\delta = \opnorm{\Appext}{S^\perp}{\Vext}$. Given $s^\perp \in S^\perp$, we write $s^\perp = v + s$ with $v \in V$ and $s \in S$ and observe that
\begin{equation*}
 s^\perp = s^\perp - \Ritz_S s^\perp = v - \Ritz_S v
\quad\text{and}\quad
 \Appext s^\perp = \App v - \Ritz_S v.
\end{equation*}
Hence $\delta = \delta_V$ and
\begin{equation}
\label{Pext>delta_V}
 \opnorm{\Appext}{\Vext}{\Vext} \leq \sqrt{1+\delta_V^2}.
\end{equation}  To show that this is actually an equality, note that, for any $v \in V$,
\begin{equation}
\label{orth}
 \norm{v - \Ritz_S v}^2 + \norm{\Ritz_S v - \App v}^2
 =
  \norm{v - \App v}^2
 \leq
 \opnorm{\Appext}{\Vext}{\Vext}^2 \norm{v - \Ritz_S v}^2,
\end{equation}
where we first combined the orthogonality of $\Ritz_S$ with $\Ritz_S v - \App v \in S$ and then used Theorem \ref{T:qopt}.  Rearranging terms, we see that $\delta_V^2 \leq \opnorm{\Appext}{\Vext}{\Vext}^2 - 1$, yielding the desired inequality $\sqrt{1+\delta_V^2} \leq \opnorm{\Appext}{\Vext}{\Vext}$.
\end{proof}

The following two remarks discuss the nature of $\delta_V$.

\begin{remark}[$\delta_V$ and (non)conforming consistency]
\label{R:delta_V-consistency}
In the conforming case $S \subseteq V$, without assuming the quasi-optimality of the underlying method, the existence of $\delta_V$ is equivalent to full algebraic consistency.  Therefore, $\delta_V$ can be seen as a quantitative generalization of full algebraic consistency to the nonconforming case. It measures, in relative manner, how much the method deviates from the best approximation $\Ritz_S$. Thus, Proposition \ref{P:consistency-with-stability} is a specification of the second Strang lemma, where the exploitation of the nonconforming direction is compared with the best approximation error. Let us illustrate this in the purely nonconforming case $V \cap S = \{0\}$. The best case corresponds to $\App = \Ritz_S$, yielding $\delta_V = 0$ and $\Cqopt = 1$. Instead, $\App = 0$ is  quasi-optimal with $\delta_V = ( \inf_{\norm{s}=1} \inf_{\Ritz_S v = s} \norm{s-v}  )^{-1}$, which becomes infinity as the distance between $S$ and $V$ tends to $0$.
\end{remark}

\begin{remark}[$\delta_V$ and stability]
\label{R:delta_V-stability}
The size of $\delta_V$ is in general affected by stability.  Indeed, using \eqref{Pext>delta_V}, we readily derive
\begin{equation*}
 \delta_V
 \geq
 \sqrt{\opnorm{\Appext}{\Vext}{\Vext}^2 - 1}
 \geq
 \sqrt{\opnorm{\App}{V}{S}^2 - 1}
 =
 \sqrt{ \Cstab^2 - 1}
\end{equation*}
and notice in particular that, if a sequence of methods becomes unstable, the corresponding $\delta_V$'s become unbounded.
\end{remark}

We now turn to the second option of applying Lemma \ref{L:restrictions}. Interestingly, it provides an alternative consistency measure which is essentially independent of stability.

\begin{proposition}[Consistency without stability]
\label{P:consistency-without-stability}
Let $\Ritz_V$ be the $\aext$-orthogonal projection onto $V$ and $\delta_S \geq 0$ be the smallest constant such that
\begin{equation*}
 \forall s \in S
\quad
 \norm{s - P\Ritz_V s}
 \leq
 \delta_S
 \norm{s - \Ritz_V s}.
\end{equation*}
Then the quasi-optimality constant satisfies
\begin{equation}
\label{deltaS-bounds}
 \max \{ \Cstab, \delta_S \}
 \leq
 \Cqopt
 \leq
 \sqrt{ \Cstab^2 + \delta_S^2}.
\end{equation}
\end{proposition}

\begin{proof}
Thanks to Theorem \ref{T:qopt}, it suffices to apply Lemma \ref{L:restrictions} with $H = \Vext$, $T = \Appext$ and $Y = V$ and to observe the following identities: given $v^\perp \in V^\perp$, $v \in V$, $s \in S$ such that $v^\perp = v + s$, we have
\begin{equation*}
v^\perp = v^\perp - \Ritz_V v^\perp = s - \Ritz_V s
\quad\text{and}\quad
\Appext v^\perp = s - \App\Ritz_V s.
\qedhere
\end{equation*}
\end{proof}

We now discuss also the nature of $\delta_S$, elaborating its differences from the first consistency measure $\delta_V$.

\begin{remark}[$\delta_S$ and (non)conforming consistency]
\label{R:delta_S-consistency}
As for $\delta_V$, the existence of $\delta_S$ is equivalent to full algebraic consistency in the conforming case $S \subseteq V$. Correspondingly, it can be seen as an alternative, quantitative generalization of full algebraic consistency to the nonconforming case. The alternative $\delta_S$ is however not comparing with the best approximation $\Ritz_S$. In particular, we have that $\delta_S=0$ implies
\begin{equation*}
 \Cqopt
 =
 \opnorm{\Appext}{\Vext}{\Vext}
 =
 \opnorm{\App}{V}{S}
 =
 \Cstab,
\end{equation*}
which is an interesting property not involving the best approximation $\Ritz_S$.
Let us illustrate how the difference is expressed in measuring the exploitation of the nonconforming directions by considering, as in Remark \ref{R:delta_V-consistency}, the purely nonconforming case $V \cap S = \{ 0\}$. Here the best choice $\App = \Ritz_S$ leads to $\delta_S<1$, while $\App = 0$ gives $\delta_S = (\inf_{\norm{s}=1} \norm{s-\Ritz_V s})^{-1}$. In the latter case, $\delta_S$ like $\delta_V$ becomes infinity as the distance between $S$ and $V$ tends to $0$, although in a (possibly) other manner. 
\end{remark}

\begin{remark}[$\delta_S$ and stability]
\label{R:delta_S-stability}
We illustrate that the quantities $\delta_S$ and $\Cstab$ are essentially independent. In order to make sure that this is not affected by a possible lack of approximability, we consider the following setting with a sequence of discrete spaces:
\begin{gather*}
 \Vext = \ell_2(\R)
 \text{ with canonical basis }(e_i)_{i=0}^\infty,
\quad
 \aext(v,w) = \sum_{i=0}^\infty v_i w_i,
\intertext{where we identify $v = \sum_{i=0}^\infty v_i e_i$ with $(v_i)_{i=0}^\infty$, etc., and}
 V = \overline{\text{span}\,\{ e_i \mid i \geq 1 \}},
\quad
 S_n = \text{span}\,\{ e_i \mid i=1,\dots,n-1 \} + \text{span}\,\{ \alpha_n e_0 + e_n \},
\end{gather*}
where $n \geq 1$ and $(\alpha_n)_n \subseteq \R_+$ is some sequence of positive reals. Here only $\alpha_n e_0 + e_n$ is nonconforming and thus not involved in full algebraic consistency. If $\lim_{n\to\infty} \alpha_n = 0$, this direction becomes a new conforming direction, while for $\lim_{n\to\infty} \alpha_n = \infty$, it gets orthogonal to $V$. In any case, we have
\begin{equation*}
 S_n \cap V = \text{span}\,\{e_i \mid i = 1,\dots, n-1 \}
\quad\text{and}\quad
 V = \overline{ \bigcup_{n \geq 1} S_n }.
\end{equation*}
Moreover, straight-forward computations reveal that the orthogonal projections onto $S_n$ and $V$ are given by
\begin{gather*}
 \Ritz_{S_n} v
 =
 \sum_{i=1}^{n-1} v_i e_i + \frac{v_n}{1+\alpha_n^2} (\alpha_n e_0 + e_n)
\text{ for } v \in V,
\quad
 \Ritz_V s = \sum_{i=1}^{n} s_i e_i
\text{ for } s \in S.
\end{gather*}

One possibility to deal with the nonconforming direction $\alpha_n e_0 + e_n$ is to ignore it, e.g., by choosing methods with the approximation operators
\begin{equation*}
 \App_{1,n} v = \sum_{i=1}^{n-1} v_i e_i
\quad\text{for}\quad
 v \in V.
\end{equation*}
Each approximation operator $\App_{1,n}$ is fully algebraically consistent and fully stable with $\opnorm{\App_{1,n}}{V}{S} = 1$.  Furthermore, $\Ritz_V(\alpha_n e_0 + e_n) = e_n$ and $\App_{1,n} e_n = 0$ yield
\begin{equation*}
 \delta_{S_n} 
 \geq
 \frac{ \norm{\overline{s}_n - \App_{1,n} \Ritz_V \overline{s}_n} }%
  { \norm{\overline{s}_n - \Ritz_V \overline{s}_n} }
 =
 \frac{\norm{\overline{s}_n}}{\alpha_n \norm{e_0}}
 =
 \frac{\sqrt{1+\alpha_n^2}}{\alpha_n}
 \geq
 \frac{1}{\alpha_n}.
\end{equation*}
with $\overline{s}_n := \alpha_n e_0 + e_n$.  Consequently, letting $\alpha_n\to 0$ shows that $\delta_S$ can become arbitrarily large, while the stability constant attains its minimal value for the case $S_n\cap V \neq \{0\}$.

Given a sequence $(\beta_n)_n \subseteq \R_+$ of positive reals, the approximation operators
\begin{equation*}
 \App_{2,n} v
 :=
 \sum_{i=1}^{n-1} v_i e_i
  + \left( v_n + \frac{\beta_n}{1+\alpha_n^2} v_{n+1}\right) (\alpha_n e_0 + e_n)
\quad\text{for}\quad v \in V 
\end{equation*}
exploit the nonconforming direction $\alpha_n e_0 + e_n$. Again, each $\App_{2,n}$ is fully algebraically consistent and fully stable.  Here, since $\App_{2,n} \Ritz_V s = s$ for all $s \in S$, we have that $\delta_S = 0$, while
\begin{equation*}
 \opnorm{\App_{2,n}}{V}{S}
 \geq
 \frac{ \norm{ \App_{2,n} e_{n+1} } }{ \norm{e_{n+1}} }
 \geq
 \frac{ \beta_n }{ \sqrt{ 1 + \alpha_n^2 } }.
\end{equation*}
Thus, $ \beta_n / \sqrt{ 1 + \alpha_n^2 } \to \infty $ shows that the stability constant can become arbitrarily large, while $\delta_S$ attains its minimal value $0$.
\end{remark}

\begin{remark}[Asymptotic consistency]
\label{R:asym-consistency}
The preceding remark exemplifies that the exploitation of the nonconforming direction measured by $\delta_V$ and $\delta_S$ is relevant also `in the limit' for sequences of discrete spaces and can be controlled via the uniform boundedness of the consistency measures.
\end{remark}

We conclude this section with slight generalizations of Propositions~\ref{P:consistency-with-stability} and \ref{P:consistency-without-stability}. 

\begin{remark}[Consistency measures and non-quasi-optimality]
\label{R:consistency-infty}
If the method underlying $\App$ is not quasi-optimal, we may set $\Cqopt=\infty$. Similarly, if $\delta_V$ (or $\delta_S$) does not exist, we set $\delta_V = \infty$ (or $\delta_S=\infty$). Then
\begin{equation*}
 \delta_V = \infty \iff \Cqopt = \infty
\qquad\text{and}\qquad
 \delta_S=\infty \implies \Cqopt = \infty
\end{equation*}
and, using standard conventions for $\infty$, the formulas in Propositions~\ref{P:consistency-with-stability} and \ref{P:consistency-without-stability} hold irrespective of quasi-optimality.	
\end{remark}

\section{The structure of quasi-optimal methods}
\label{S:Structure}
%
%
% intro
As explained in the introduction, there is a great interest to devise quasi-optimal nonconforming methods. To this end, it is useful to determine the structure of nonconforming methods that are quasi-optimal. This is the task of this section, which, in light of Theorem~\ref{T:qopt}, reduces to determine the structure of full stability and full algebraic consistency.

\subsection{Extended approximation operator and extended bilinear form}
\label{S:extensions}
Our analysis of quasi-optimality in \S\ref{S:qopt} has been centered around
the extended approximation operator $\Appext$. In this subsection we relate this key tool to the extended bilinear form $\bext$ from Lemma \ref{L:ConsistencyWithExtension} and, thus, more closely to the data $(a,S,b,L)$ defining problem and method.

\begin{lemma}[Extensions of approximation operator and bilinear forms]
\label{L:Pext_and_bext}
The approximation operator $P$ extends to a bounded linear projection $\Appext$ from $\Vext$ onto $S$ if and only if there exists a bounded common extension $\bext$ of $b$ and $\langle LA\cdot,\cdot\rangle$ to $\Vext\times S$.

If one of the two extensions exists, we have the following generalization of the Galerkin orthogonality:
\begin{equation*}
%\label{gen-Galerkin-orthogonality}
%
\forall \vext \in \Vext, \sigma \in S
\quad
\bext(\vext - \Appext\vext,\sigma) = 0.
\end{equation*}

\end{lemma}

\begin{proof}
Assume $\Appext$ is a bounded linear projection from $\Vext$ onto $S$ extending $P$.  Then
\begin{equation}
\label{def_bext}
	\bext(\vext,\sigma)
	:=
	b(\Appext\vext,\sigma)
	%, \quad
	% \vext\in\Vext, \sigma\in S,
\end{equation}
defines a bounded bilinear form on $\Vext\times S$.  Since $\Appext$ is a projection onto $S$, $\bext$ is an extension of $b$.  Furthermore, if $v\in V$ and $\sigma \in S$, then $\Appext_{|V} = \App$ yields $\bext(v,\sigma) = b(\App v, \sigma) = \langle LAv,\sigma \rangle$. Consequently, $\bext$ is also an extension of $\langle LA\cdot,\cdot \rangle$.
	
% from bext to Pext
Conversely, assume that $\bext$ is a bounded common extension of $b$ and $\langle LA\cdot,\cdot \rangle$ on $\Vext \times S$. Given $\vext\in\Vext$,  define $\Appext\vext$ by
\begin{equation}
\label{def_Appext}
 \Appext\vext \in S
\quad\text{such that}\quad
	\forall \sigma \in S
	\;\;
	b(\Appext\vext,\sigma) = \bext(\vext,\sigma).
\end{equation}
Since $b$ is a nondegenerate bilinear form on $S\times S$, the element $\Appext\vext$ exists, is unique and depends on $\vext$ linearly.  The uniqueness and $\bext=b$ on $S \times S$ give $\Appext_{|S} = \id_S$.  Using $\bext = \langle LA\cdot,\cdot \rangle = b(\App\cdot,\cdot)$ on $V\times S$, we obtain $\Appext_{|V} = \App$.  Finally, the boundedness of $\bext$ entails the boundedness of $\Appext$ and the claimed equivalence is
verified.
 
It remains to verify the generalized Galerkin orthogonality. If one of the two extensions exists, then the other one is given either by \eqref{def_bext} or by \eqref{def_Appext}, which both just restate the claimed generalization.
\end{proof}

The close relationship between the two extensions $\Appext$ and $\bext$ suggests that the operator norm $\opnorm{\Appext}{\Vext}{\Vext}$ can be reformulated in terms of $\bext$. To this end, the following lemma will be very useful, which in turn exploits the following fact from linear functional analysis; see, e.g., Brezis \cite{Brezis:11}.
If $X$ and $Y$ are normed linear spaces, $T:X \to Y$ is a linear operator and $T^\star$ stands for its adjoint, then
\begin{equation}
\label{adjoint-for-bdd-op}
 T \text{ is bounded}
 \implies
 D(T^\star) = Y'
 \text{ with }
 \opnorm{T^\star}{Y'}{X'} = \opnorm{T}{X}{Y}.
\end{equation}

\begin{lemma}[$b$-duality for energy norm on $S$]
\label{L:b-duality}
The nondegenerate bilinear form $b$ induces a norm on $S$ by
\begin{equation*}
 \norm{\sigma}_{b}
 :=
 \norm{b(\cdot,\sigma)}_{S'}
 =
 \sup_{s \in S, \norm{s} = 1} b(s,\sigma),
\quad
 \sigma \in S, 
\end{equation*}
satisfying
\begin{equation*}
 \norm{s} = \sup_{\sigma\in S} \frac{b(s,\sigma)}{\norm{\sigma}_b}.
\end{equation*}
\end{lemma}

\begin{proof}
Obviously, $\norm{\cdot}_{b}$ is a seminorm and definite thanks to the nondegeneracy of $b$. To verify the claimed identity, we observe
\begin{equation}
\label{b;bddness}
 \sup_{s\in S}\sup_{\sigma\in S}
  \frac{b(s,\sigma)}{\norm{s}\norm{\sigma}_{b}}
 =
 \sup_{\sigma\in S} \sup_{s\in S}
	\frac{b(s,\sigma)}{\norm{s}\norm{\sigma}_{b}}
 =
 1
\end{equation}
and
\begin{equation}
\label{b;infsup}
 \infimum_{s \in S} \sup_{\sigma \in S}
	\frac{b(s,\sigma)}{\norm{s}\norm{\sigma}_{b}}
 =
 \infimum_{\sigma \in S} \sup_{s \in S}	
	\frac{b(s,\sigma)}{\norm{s}\norm{\sigma}_{b}} =  1,
\end{equation}
where the `=1's follow from the definition of $\norm{\cdot}_b$ and the first equality in \eqref{b;infsup} follows from \eqref{adjoint-for-bdd-op} applied to the inverse of $B$, the linear operator representing $b$. Combining \eqref{b;bddness} and \eqref{b;infsup}, we see that
\begin{equation*}
 \sup_{\sigma \in S}
  \frac{b(s,\sigma)}{\norm{s}\norm{\sigma}_{b}}
 =
 1 
\end{equation*}
for all $s\in S$ and the claimed identity is verified.
\end{proof}

\begin{lemma}[Norms of extensions]
\label{L:Pext=bext}
If one of the extensions in Lemma \ref{L:Pext_and_bext} exists, we have
\begin{equation*}
\label{||Pext||=bext}
 \opnorm{\Appext}{\Vext}{\Vext}
 =
 \sup_{\sigma\in S}
	\frac{\|\bext(\cdot,\sigma)\|_{\Vext'}}{\norm{b(\cdot,\sigma)}_{S'}}
\end{equation*}
with the `extended' dual norm
$
\norm{\ell}_{\Vext'}
:=
\sup_{\vext\in\Vext, \norm{\vext} = 1} \langle \ell,\vext \rangle.
$
\end{lemma}

\begin{proof}
Applying Lemma \ref{L:b-duality}, the generalized Galerkin orthogonality of Lemma~\ref{L:Pext_and_bext} and the definition of the extended dual norm, we infer
\begin{align*}
 \opnorm{\Appext}{\Vext}{\Vext}
 &=
 \sup_{\vext\in\Vext} \frac{\norm{\Appext\vext}}{\norm{\vext}}
 =
 \sup_{\vext\in\Vext,\sigma\in S}
	\frac{b(\Appext\vext,\sigma)}{\norm{\vext}\norm{\sigma}_b}
 =
 \sup_{\vext\in\Vext,\sigma\in S}
	\frac{\bext(\vext,\sigma)}{\norm{\vext}\norm{\sigma}_b}
\\
 &=
 \sup_{\sigma\in S}
	\frac{\norm{\bext(\cdot,\sigma)}_{\Vext'}}{\norm{\sigma}_b}
 =
 \sup_{\sigma\in S}
	\frac{\norm{\bext(\cdot,\sigma)}_{\Vext'}}{\norm{b(\cdot,\sigma)}_{S'}}.
\qedhere
\end{align*}
\end{proof}

Before closing this subsection, two remarks are in order.
\begin{remark}[Alternative proof and formula]
\label{R:alternative-formula}
An alternative proof of Lemma \ref{L:Pext=bext} may be based on a continuous counterpart of $\norm{\cdot}_b$ from Lemma \ref{L:b-duality}; see Tantardini and Veeser \cite[Theorem 2.1]{Tantardini.Veeser:16}.  Using that approach, one derives also
\begin{equation*}
 \opnorm{\Appext}{\Vext}{\Vext}
 =
 \sup_{s\in S, \norm{s}=1} \; \infimum_{\sigma \in S, \norm{\sigma}=1}
	\frac{\|\bext(\cdot,\sigma)\|_{\Vext'}}{|b(s,\sigma)|}.
\end{equation*}	
by duality.
\end{remark}

\begin{remark}[Reformulations of quasi-optimality]
\label{R:qopt}
Remarks \ref{R:FullStability} and \ref{R:QuasiOptRequiresFullConsistency},	Lem\-ma\-ta~\ref{L:Pext} and \ref{L:Pext_and_bext} as well as Theorem \ref{T:qopt} show that the following statements are equivalent reformulations of quasi-optimality for a nonconforming method $\app=(S,b,L)$ with approximation operator $\App$:
\begin{subequations}
\label{char-qopt}
\begin{align}
 \label{char:fs-fc} % \tag{i}
 &\text{$\app$ is fully algebraically consistent and fully stable.}
\\ \label{char:P} %\tag{ii}
 &\text{$Ps=s$ for all $s \in S\cap V$ and $P$ is bounded.}
\\ \label{char:Pext} %\tag{iii}
 &\text{$P$ extends to a linear projection $\Appext$ from $\Vext$ onto $S$ that is bounded.}
\\ \label{char:bext} %\tag{v}
 &\text{$b$ and $\langle LA\cdot,\cdot\rangle$ have a common extension $\bext$ that is bounded.}
\\
 &\text{$\App$ is bounded and $b,\App$ have extensions $\bext,\Appext$ such that $\bext(\vext-\Appext\vext,\sigma)=0$}
\\ \nonumber &\text{for all $\vext\in\Vext$ and $\sigma \in S$.}
\end{align}
\end{subequations}
It is worth observing that no additional regularity beyond the natural one in \eqref{ex-prob} is involved.  All this illustrates that extensions, as developed in our approach, are a well-tuned tool in the analysis of the quasi-optimality of nonconforming methods. 
\end{remark}

\subsection{The structure of full stability}
\label{S:stab-smooting}
%
% intro
In this subsection we determine the structure of nonconforming methods that are fully stable. 

% preliminaries
To this end, \eqref{adjoint-for-bdd-op} and the following facts of linear functional analysis will be basic: if $X$ and $Y$ are normed linear spaces and $T:X\to Y$ linear, then
\begin{gather}
\label{dim-and-bdd}
 \dim X < \infty \iff \text{all linear operators $X \to Y$ are bounded},
\\
\label{T*surjective-iff-Tinjective}
 \text{if $\dim X < \infty$, then }
 T^\star \text{ surjective} \iff T \text{ injective}, 
\end{gather}
see, e.g., \cite{Brezis:11} and \cite[p.\ 1418]{Buckholtz:00}.

% setting
Let $\app=(S,b,L)$ be a nonconforming method and recall that $\app$ is fully stable if and only if the operator $\app:V' \to S$ is bounded, where $V'$ and $S$ are equipped, respectively, with the dual and extended energy norm.

% role of L
We claim that the full stability of $\app$ hinges on the boundedness of $L$.  In light of Remark~\ref{R:FullStability}, we may assume that $D(\app)=D(L)=V'$.  The equivalence \eqref{dim-and-bdd} yields the following two consequences. First, the boundedness of $M:V'\to S$ is a true requirement, because its domain $V'$ has infinite dimension. Second, the critical operator in the composition $\app = B^{-1} L$ from \eqref{M=} is $L$.  In fact, its domain $V'$ has infinite dimension, while the domain $S'$ of $B^{-1}$ has finite dimension. Consequently, a method $M$ is fully stable if and only if it is entire and the operator $L:V' \to S'$ is bounded.

% characterization of L bounded
Next, we characterize the class of bounded linear operators from $V'$ to $S'$ and derive first a necessary condition.  Let $L:V' \to S'$ be linear and bounded.  Owing to \eqref{adjoint-for-bdd-op}, its adjoint $L^\star$ is a bounded linear operator from $S''$ to $V''$.  Since the spaces $S$ and $V$ are reflexive, we thus deduce the existence of a linear operator $E:S\to V$ such that
\begin{equation}
\label{L=E*}
 \forall \ell \in V', \sigma \in S
\quad
 \left\langle L\ell, \sigma\right\rangle 
 =
 \left\langle \ell, E\sigma\right\rangle .
\end{equation}
Conversely, if $E:S\to V$ is a linear operator satisfying \eqref{L=E*}, then $L$ is bounded on $V'$ with $\opnorm{L}{V'}{S'} = \opnorm{E}{S}{V}$ by \eqref{adjoint-for-bdd-op} and \eqref{dim-and-bdd}.
%Finally, since $\dim S < \infty$, we seem from \eqref{T*surjective-iff-Tinjective} that $L$ is surjective if and only if $E$ is injective.

%
\begin{remark}[Smoothing of $E$]
\label{R:smoothing}
Usually, the nonconformity $S \not\subseteq V$ arises from a lack of smoothness, e.g., across interelement boundaries in the case of finite element methods.  The operator $E:S \to V$ may then be viewed as a smoothing operator.
\end{remark}

The above observations prepare the following result, which is our first step towards the structure of quasi-optimal methods.

\begin{theorem}[Full stability and smoothing]
\label{T:fstab-smoothing}
A nonconforming method $\app = (S,b,L)$ for \eqref{ex-prob} is fully stable if and only if $L$ is the adjoint of a linear smoothing operator $E:S \to V$.
	
The discrete problem for $\ell \in V'$ then reads
\begin{equation}
\label{disc-prob-with-smoothing}
 \forall \sigma \in S
\quad
	b(\app\ell, \sigma)
	=
	\left\langle \ell, E\sigma\right\rangle
\end{equation} 
and the stability constant satisfies
\begin{equation}
\label{Cstab-with-smoothing}
 \Cstab
 =
 \opnorm{\app}{V'}{S}
 =
 \sup_{\sigma\in S} \, \frac{\norm{E\sigma}}{\norm{b(\cdot,\sigma)}_{S'}}.
\end{equation}
Moreover, the range of $M$ is $S$ if and only if $E$ is injective.
\end{theorem} 

\begin{proof}
The observations preceding Theorem \ref{T:fstab-smoothing} show that $\app$ is fully stable if and only if $L$ is the adjoint of a linear smoothing operator $E:S\to V$. Moreover, they provide the claimed form of the discrete problem via \eqref{L=E*}. The second equivalence readily follows from \eqref{T*surjective-iff-Tinjective} and Remark \ref{R:surjectivity-of-L}.

To verify \eqref{Cstab-with-smoothing}, we combine Lemma \ref{L:b-duality} with $\norm{v} = \sup_{\ell\in V', \norm{\ell}_{V'}=1} \langle \ell,v \rangle$, see, e.g., Brezis \cite[Corollary 1.4]{Brezis:11}:
	\begin{align*}
	\Cstab
	&=
	\opnorm{\app}{V'}{S}
	=
	\sup_{\ell\in V'} \frac{\norm{\app\ell}}{\norm{\ell}_{V'}}
	=
	\sup_{\ell \in V', \sigma \in S}
	\frac{b(\app\ell,\sigma)}{\norm{\ell}_{V'}\norm{\sigma}_b}
	\\
	&=
	\sup_{\sigma \in S, \ell \in V'}
	\frac{\langle\ell,E\sigma\rangle}{\norm{\ell}_{V'}\norm{\sigma}_b}
	=
	\sup_{\sigma \in S}
	\frac{\norm{E\sigma}}{\norm{\sigma}_b}
	=
	\sup_{\sigma \in S}
	\frac{\norm{E\sigma}}{\norm{b(\cdot,\sigma)}_{S'}}.
	\qedhere
	\end{align*}
\end{proof}

Let us start the discussion of this result by considering a canonical choice for the smoother $E$.
\begin{remark}[Trivial smoothing for conforming methods]
\label{R:trivial-smoothing-cG}
Assume that the discrete space $S\subseteq V$ is conforming and consider the simplest choice $E=\id_S$. For this classical case, \eqref{Cstab-with-smoothing} reduces to the well-known identity
\begin{equation*}
	\Cstab
	=
	\sup_{\sigma \in S} \frac{\norm{\sigma}}{\norm{b(\cdot,\sigma)}_{S'}}
	=
	\left(
	\infimum_{\sigma \in S} \sup_{s \in S} \frac{b(s,\sigma)}{\norm{s}\norm{\sigma}}
	\right)^{-1}
	=
	\left(
	\infimum_{s \in S} \sup_{\sigma \in S} \frac{b(s,\sigma)}{\norm{s}\norm{\sigma}}
	\right)^{-1}.
\end{equation*}
\end{remark}

\begin{remark}[Failure of $\id_S$]
\label{R:failure-of-idS}
Let $S$ be a nonconforming discrete space with $S \not\subseteq V$.  Then the choice $E=\id_S$ is not compatible with full stability and so, in view of Theorem \ref{T:qopt}, not with quasi-optimality.  Indeed, Theorem \ref{T:fstab-smoothing} shows that $E(S) \subseteq V$ is necessary for full stability. Consequently, the condition $Es=s$ entails $s \in S\cap V$ and thus produces a contradiction for any $s\in S\setminus V$. We therefore need to define $Es$ for $s\in S\setminus V$ differently, which, in view of the nature of $S$ and $V$ in applications, typically amounts to some kind of smoothing.  
\end{remark}

Most DG methods and classical NCFEM rely on the simple choice $E=\id_S$, requiring that the load term $\ell$ in \eqref{ex-prob} has some additional regularity. Remark \ref{R:failure-of-idS} implies that these methods are not fully stable and so, in view of Theorem~\ref{T:qopt}, not quasi-optimal. This provides an alternative to falsify quasi-optimality with Remark~\ref{R:qopt->entire}.

We end this subsection by considering first alternatives to $E=\id_S$ and illustrating that the choice of $E$ is in general a delicate matter.
\begin{remark}[Previous uses of smoothing]
\label{R:Previous-smoothers}
Advantages of suitable smoothing have been previously observed. An obvious one is that the method can be made entire and this has been pointed out, e.g., in the DG context by Di Pietro and Ern \cite{Ern.DiPietro:12}.

Comparing the Hellan-Hermann-Johnson method with the Morley method, Ar\-nold and Brezzi \cite{Arnold.Brezzi:85} showed that a particular smoothing in the Morley method leads to an a~priori error estimate requiring less regularity of the underlying load term. This corresponds to an increased stability thanks to the employed smoothing.

Also in the context of fourth order problems, Brenner and Sung \cite{Brenner.Sung:05} proposed $C^0$ interior penalty methods and proved a~priori error estimates also for nonsmooth loads. Furthermore, the involved regularity is minimal from the viewpoint of approximation.

Finally, Badia et al.\ \cite{Badia.Codina.Gudi.Guzman:14} used a rather involved smoother, which is related to our construction in \cite{Veeser.Zanotti:17p2}, to show a partial quasi-optimality result.
\end{remark}

\begin{remark}[Smoothers into $S\cap V$]
\label{R:qo-Lsurjective}
It may look natural to use smoothers $E$ that map into the conforming part $S\cap V$ of the discrete space.  In view of Remark \ref{R:surjectivity-of-L}, the range $R(M)$ of the corresponding method is a proper subspace of $S$, whenever $S \setminus V \neq \emptyset$. Quasi-optimality is then not ruled out, but it hinges on the validity of results like Corollary 1 in Veeser \cite{Veeser:16} and requires in particular that $S \cap V$ is not small.
\end{remark}

\begin{remark}[Optimal smoothing]
\label{R:Optimal-smoothing}
The structure of full stability does not principally exclude methods that are optimal from the viewpoint of approximation. Consequently, the variational crime of nonconformity does not necessarily result in some consistency error. 
To see this, consider the discrete bilinear form $b = \aext_{|S\times S}$.  Since
\begin{equation*}
 \forall v \in V, \sigma \in S
\quad
 \aext(\App v - v, \sigma)
 =
 \aext(v, E\sigma - \sigma),
\end{equation*}
we have
\begin{equation*}
 \App = \Ritz_S
 \iff
 E = \Ritz_V.
\end{equation*} 
In other words: a nonconforming method $(S,\aext_{S\times S}, E^\star)$ provides the best approximation if and only if the smoother $E$ is the $\aext$-orthogonal projection onto $V$.  This smoother is however not feasible in the sense of the following remark.
\end{remark}

\begin{remark}[Feasible smoothing]
\label{R:comput-E}
Adopt the notation of Remark \ref{R:ComputingDiscreteSolutions} and let $\varphi_1,\dots,\varphi_n$ be a computionally convenient basis for the discrete bilinear form $b$.  In order to compute $\app\ell$ by \eqref{disc-prob-with-smoothing} with optimal complexity, the total number of operations for evaluating $\langle \ell, E\varphi_i\rangle$ for all $i=1,\dots,n$ has to be of order $O(n)$.  A sufficient condition for this is that, for each $i=1,\dots,n$,  the function $E\varphi_i$ is locally supported so that $\langle \ell, E\varphi_i\rangle$ can be evaluated at cost $O(1)$. 
\end{remark}

\subsection{The structure of quasi-optimality}
\label{S:qopt-smoothing}
%
% intro 
We are finally ready for the main results of our abstract analysis about the quasi-optimality of nonconforming methods. 

\begin{theorem}[Quasi-optimality and smoothing]
\label{T:qopt-smoothing}
A nonconforming method $\app=(S,b,L)$ for \eqref{ex-prob} is quasi-optimal if and only if there exists a linear smoothing operator $E:S \to V$ such that the discrete problem reads
\begin{equation*}
 \forall \sigma \in S
\quad
 b(M\ell,\sigma) = \langle \ell, E\sigma \rangle
\end{equation*}
for any $\ell \in V'$ and 
\begin{equation}
\label{fa-consistency-with-E}
 \forall u \in S\cap V, \sigma \in S
\quad
 b(u, \sigma)
 =
 a(u, E\sigma).  
\end{equation}
Its quasi-optimality constant is given by
\begin{equation}
\label{Cqopt=;smoothing}
 \Cqopt
 =
 \sup_{\sigma \in S} \,
  \frac{ \sup_{\norm{v+s} = 1} a(v,E\sigma) + b(s,\sigma) }
   {\sup_{\norm{s} = 1} b(s,\sigma)},
\end{equation}
where $v$ varies in $V$ and $s$ in $S$.
\end{theorem}

\begin{proof}
We first check the claimed equivalence.  The form of the discrete problem means that $L$ is the adjoint of $E$ and, in view of Theorem \ref{T:fstab-smoothing}, that $M$ is fully stable. Moreover, since	
\begin{equation}
\label{LA-e-E}
\langle LAu, \sigma \rangle
=
\langle Au,E\sigma \rangle
=
a(u, E\sigma)
\end{equation}
for all $u\in V$ and $\sigma \in S$, \eqref{fa-consistency-with-E} is equivalent to \eqref{Consistency}, i.e. full algebraic consistency.  Consequently, the claimed equivalence follows from Theorem~\ref{T:qopt}.

To show the identity for the quasi-optimality constant, we observe that the extension $\bext$ exists and satisfies, for $\vext\in\Vext$, $v \in V$, $s,\sigma \in S$ such that $\vext = v + s$,
\begin{equation*}
 \bext(\vext, \sigma)
 =
 \langle LAv,\sigma \rangle + b(s,\sigma)
 =
 a(v, E\sigma) +  b(s,\sigma)  
\end{equation*}
thanks to \eqref{LA-e-E}.  Therefore, the formula for $\Cqopt$ follows from Theorem \ref{T:qopt} and Lemma \ref{L:Pext=bext}.
\end{proof}

We start the discussion of Theorem \ref{T:qopt-smoothing} by a remark about the notion of Galerkin methods.

\begin{remark}[Galerkin methods]
\label{R:nG}
Assume first that the discrete space $S\subseteq V$ is conforming. Then trivial smoothing $E=\id_S$ in \eqref{fa-consistency-with-E} yields $b = a_{|S \times S}$. In other words: conforming Galerkin methods are the only quasi-optimal methods with the simplest choice $E=\id_S$ for smoothing. 

Next, consider a general nonconforming discrete space $S$, together with the simplest choice for smoothing in the conforming part $S\cap V$, i.e.\ with $E_{|S\cap V} = \id_{S\cap V}$. Here \eqref{fa-consistency-with-E} yields $b_{|S_C \times S_C} = a_{|S_C \times S_C}$ with $S_C = S \cap V$. Thus, nonconforming Galerkin methods are the only candidates for quasi-optimal methods with $E_{|S\cap V} = \id_{S\cap V}$. In this context, the following observation if useful in constructing $E$ with $E_{|S\cap V} = \id_{S\cap V}$. If $E$ maps some $s \in S\setminus V$ in $S\cap V$, then the injectivity of $E$ is broken and, in view of Theorem \ref{T:fstab-smoothing}, the range of the method is a strict subspace of $S$.
\end{remark}

\begin{remark}[Comparison with second Strang lemma]
\label{R:2nd-Strang}
For conforming Galerkin methods, Theorem \ref{T:qopt-smoothing} reduces to the well-known C\'ea lemma, with $\Cqopt=1$.  C\'ea's lemma is a basic building block in the analysis of the energy norm error for conforming methods. In the context of nonconforming methods, the second Strang lemma is often used as a replacement.  Theorem \ref{T:qopt} provides a specialization revealing the structure of quasi-optimal methods and so lays the groundwork for their design.
\end{remark}

\begin{remark}[Comparison with conforming Petrov-Galerkin methods]
\label{R:Petrov-Galerkin}
Our setting of \S\ref{S:setting} includes the application of Petrov-Galerkin methods to \eqref{ex-prob}.  It is therefore of interest to compare formula \eqref{Cqopt=;smoothing} with its conforming counterpart in Theorem~2.1 of Tantardini and Veeser \cite{Tantardini.Veeser:16}:
\begin{equation*}
 \Cqopt
 =
 \sup_{\sigma \in S} \,
  \frac{ \sup_{\norm{v} = 1} b(v,\sigma) }{\sup_{\norm{s} = 1} b(s,\sigma)},   
\end{equation*}
where here $b$ stands for the continuous (and discrete) bilinear form, $v$, $s$, and $\sigma$ vary, respectively, in the continuous trial space, in the discrete trial space and in the discrete test space.  We see that \eqref{Cqopt=;smoothing} generalizes this formula, replacing the continuous bilinear form by the extended one, which interweaves discrete and continuous problems.
\end{remark}

\begin{remark}[`Classical' bound for quasi-optimality constant]
\label{R:Cqopt}
A consequence of the formula for the quasi-optimality constant in Theorem \ref{T:qopt-smoothing} and \eqref{adjoint-for-bdd-op} is the following upper bound:
\begin{equation}
\label{Cqopt<=b(ext)}
 \Cqopt
 \leq
 \frac{C_{\bext}}{\beta}
\end{equation}
with the continuity and inf-sup constants
\begin{equation*}
 C_{\bext} := \sup_{\norm{v+s}=1, \norm{\sigma} = 1} a(v,E\sigma) + b(s,\sigma),
\qquad
 \beta := \infimum_{\norm{s} = 1} \sup_{\norm{\sigma} = 1} b(s,\sigma),
\end{equation*}
where $v$ varies in $V$ and $s$ and $\sigma$ in $S$. This upper bound has the classical form of constants appearing in quasi-optimality results, apart from the slight difference that the continuity constant of the numerator involves the extended bilinear form; see also Remark \ref{R:Petrov-Galerkin}.

It is worth mentioning that the right-hand side of \eqref{Cqopt<=b(ext)} can become arbitrarily large, while its left-hand side remains bounded; see \cite[Remark~2.7]{Veeser.Zanotti:17p2}.
\end{remark}

Let us now assess what determines the size of the quasi-optimality constant.

\begin{theorem}[Size of quasi-optimality constant]
\label{T:Cqopt}
Assume $M=(S,b,L)$ is a quasi-optimal nonconforming method with linear smoother $E:S \to V$ and stability constant $\Cstab$. The consistency measure $\delta_V$ of Proposition \ref{P:consistency-with-stability} is finite and is
\begin{equation}
\label{deltaV=}
 \delta_V
 =
 \sup_{v \in V, \Ritz_S v \neq v} \, \sup_{\sigma \in S} \,
  \frac{ b(\Ritz_S v, \sigma) - a(v, E\sigma)}
   {\norm{ \Ritz_S v - v} \norm{b(\cdot,\sigma)}_{S'}}.
\end{equation}
Similarly, the consistency measure $\delta_S$ of Proposition \ref{P:consistency-without-stability} is finite and the smallest positive constant such that
\begin{equation*}
 \forall s \in S
\quad
 \sup_{\sigma \in S}
  \frac{ b(s,\sigma) - a(\Ritz_V s, E\sigma) }{ \norm{b(\cdot,\sigma)}_{S'}}
 \leq
 \delta_S
 \norm{ s - \Ritz_V s}.
\end{equation*}
Then the quasi-optimality constant of $M$ satisfies
\begin{equation*}
 \max \{\Cstab, \delta_S \}
 \leq
 \Cqopt
 =
 \sqrt{1+\delta_V^2}
 \leq
 \sqrt{\Cstab^2+\delta_S^2}.
\end{equation*}
\end{theorem}

\begin{proof}
Lemma \ref{L:b-duality} readily yields the identities
\begin{equation*}
 \norm{\Ritz_S v - \App v}
 =
 \sup_{\sigma \in S} \frac{ b(\Ritz_S v - \App v, \sigma) }{ \norm{\sigma}_b }
\quad\text{and}\quad
 \norm{ s - \App\Ritz_V s}
 =
 \sup_{\sigma \in S} \frac{ b(s - \App\Ritz_V s, \sigma) }{ \norm{\sigma}_b }.
\end{equation*}
Notice also $b(\App v, \sigma) = b(\app Av, \sigma) = \langle LAv, \sigma \rangle = a(v,E\sigma)$ and $\norm{\sigma}_b = \norm{b(\cdot,\sigma)}_{S'}$ for $v \in V$ and $\sigma \in S$ as well as $V \setminus S \neq \emptyset$. Therefore, $\delta_V$ and $\delta_S$ coincide with the corresponding quantities in Propositions \ref{P:consistency-with-stability} and \ref{P:consistency-without-stability} and Theorem \ref{T:Cqopt} just restates their conclusions.
\end{proof}

We refer to \S\ref{S:Cqopt} for a discussion of the relationship between $\Cqopt$ and $\Cstab$ and in particular the consistency measures $\delta_V$ and $\delta_S$. Let us further connect the expression of $\delta_V$ in this theorem with classical consistency. 

\begin{remark}[$\delta_V$ and classical consistency error]
\label{R:consistency error}
The numerator of \eqref{deltaV=} represents the action of a linear functional on $S$, namely
\begin{equation*}
b(\Ritz_S v, \sigma) - a(v, E \sigma)
=
\left\langle B \Ritz_S v - LA v, \sigma \right\rangle
=: 
\left\langle \rho, \sigma \right\rangle.  
\end{equation*}
Let us recall that $LAv$ is the discrete load associated to $v$ in problem \eqref{disc-prob} and $B\Ritz_S v$ is the linear functional obtained from the representative $\Ritz_S v$ of $v$ in $S$, through the isomorphism $B$. Introducing the norm $\norm{\cdot}_{S',b} := \sup_{\norm{b(\cdot, \sigma)}_{S'} =1} \left\langle \cdot, \sigma \right\rangle$, the quantity $\norm{\rho}_{S', b}$ is a consistency error in the sense of Arnold \cite{Arnold:15}. The measure $\delta_V$ compares this quantity with the natural benchmark in the context of quasi-optimality, i.e.\ the best error $\norm{v-\Ritz_S v}$.    
\end{remark}

%\medskip
Given $S$ and $b$, Theorem \ref{T:qopt-smoothing} reduces the construction of quasi-optimal nonconforming methods to the choice of a computationally feasible linear smoother $E$ and Theorem \ref{T:Cqopt} shows how the smoother $E$ affects the size of the quasi-optimality constant. In the follow-ups \cite{Veeser.Zanotti:17p2,Veeser.Zanotti:17p3} of this work, we devise such smoothers for various nonconforming finite element spaces.  Modifying classical NCFEM (like the Crouzeix-Raviart method), we can obtain $\delta_S=0$ and so $\Cqopt = \Cstab$, as for conforming Galerkin methods. Also DG and $C^0$ interior penalty methods can be modified to be quasi-optimal. Remarkably, additional terms not affecting full algebraic consistency entail $\delta_S > 0$ for the employed smoothing.

%
%
%\bibliographystyle{siam}
%\bibliography{Bibliografia}
%\bibliography{/home/veeser/Unimibox/myTeX/av.bib}

%
%
%
\end{document}